\documentclass[12pt,reqno, oneside]{amsart}
\usepackage{hyperref}
\hypersetup{colorlinks,citecolor=blue}

\newcommand{\seqnum}[1]{\href{http://oeis.org/#1}{{#1}}}

\newcommand{\f}{\alpha}
\newcommand{\g}{\beta}
\newcommand{\fx}{f}

\usepackage{xcolor}

\usepackage{url}
\usepackage{geometry}

\newtheorem{theorem}{Theorem}
\newtheorem{lemma}[theorem]{Lemma}
\newtheorem*{lagrange1}{Lagrange Inversion Theorem, First Form}
\newtheorem*{lagrange2}{Lagrange Inversion Theorem, Second Form}

\newcommand{\half}{\tfrac12}

\renewcommand{\th}{\tfrac32}

\begin{document}
\title[A short proof of the Deutsch-Sagan congruence]
{A short proof of the Deutsch-Sagan congruence for connected noncrossing graphs
}
\author{Ira M. Gessel$^*$}
\address{Department of Mathematics\\
   Brandeis University\\
   Waltham, MA 02453}
\email{gessel@brandeis.edu}
\date{March 29, 2014}
\thanks{$^*$This work was partially supported by a grant from the Simons Foundation (\#229238 to Ira Gessel).}

\maketitle

\bigskip


\section{Introduction}

 Let $N_n$ be the number of connected noncrossing graphs on $n$ vertices. Flajolet and Noy \cite{f-n} showed that for $n\ge2$,
\begin{equation}
\label{e-N}
N_n = \frac{1}{n-1}\sum_{i=n-1}^{2n-3}\binom{3n-3}{n+i}\binom{i-1}{i-n+1}.
\end{equation}
These numbers are sequence \seqnum{A007297} of the On-Line Encyclopedia of Integer Sequences \cite{oeis}.
Here is a table of small values of $N_n$.
\[\vbox{\halign{\ \hfil\strut$#$\hfil\ \vrule&&\hfil\ $#$\ \hfil\cr
n&1&2&3&4&5&6&7&8&9\cr
\noalign{\hrule}
N_n&1&1&4&23&156&1162&9192&75819&644908\cr}}
\]

Deutsch and Sagan \cite{ds} conjectured that 
\begin{equation}
\label{e-DS}
N_n\equiv
\begin{cases} 1\!\pmod 3,&\text{if $n$ is a power of 3 or twice a power of 3},\\
2\!\pmod 3,&\text{if $n$ is a sum of two distinct powers of 3},\\
0\!\pmod 3, &\text{otherwise}.
\end{cases}
\end{equation}

They noted that the first two cases are not hard to prove using Lucas's theorem for the residue of a binomial coefficient modulo a prime. A complicated proof of the Deutsch-Sagan conjecture was given by Eu, Liu, and Yeh \cite{ely}.

Here we give a simpler proof of Deutsch and Sagan's conjecture using Lagrange inversion.
We then discuss some numbers related to the $N_n$ that arose in Eu, Liu, and Yeh's proof, given by sums similar to \eqref{e-N} and then we show how these sums can be evaluated explicitly. 

\section{Proof of the congruence}
\label{s-cong}

We start by representing $N_n$ as a coefficient of a power series. We use the notation $[x^m] u(x)$ to denote the coefficient of $u(x)$ in the power series $u(x)$.
\begin{lemma}
\label{l-1}
For $n\ge1$, 
\begin{equation}
\label{e-N1}
N_{n+1}=\frac{1}{n}[x^{n-1}]\frac{1}{(1-x)^{n+2}(1-2x)^{n}}.
\end{equation}
\end{lemma}
\begin{proof}
We rewrite \eqref{e-N} as 
\begin{equation}
\label{e-N2}
N_{n+1}=\frac{1}{n} \sum_{k=0}^{n-1} \binom{3n}{n-1-k}\binom{n+k-1}{k}.
\end{equation}
We have
\[1-2x= (1-x)\left(1-\frac{x}{1-x}\right),\]
so
\begin{align*}
\frac{1}{(1-x)^{n+2}(1-2x)^{n}}
  &=(1-x)^{-2n-2}\left(1-\frac{x}{1-x}\right)^{-n}\\
  &=\sum_{j=0}^\infty \binom{n+j-1}{j} \frac{x^j}{(1-x)^{2n+2+j}}\\
  &=\sum_{j=0}^\infty \binom{n+j-1}{j}\sum_{k=0}^\infty 
  \binom{2n+1+j+k}{k} x^{j+k}.
\end{align*}
Then the coefficient of $x^{n-1}$ is 
\begin{equation*}
\sum_{j=0}^{n-1} \binom{n+j-1}{j} \binom{3n}{n-1-j},
\end{equation*}
which by \eqref{e-N2} is $nN_{n+1}$.
\end{proof}

Next, recall the following form of Lagrange inversion \cite[p.~42, equation (5.65)]{ec2}.
\begin{lagrange1}
Let $G(t)$ be a formal power series and let $f=f(x)$ be the unique formal power series satisfying $f=xG(f)$. Then for any formal power series $\Phi(x)$,
\begin{equation*}
[x^n] \Phi( f) = \frac{1}{n} [x^{n-1}] \Phi'(x)G(x)^n.
\end{equation*}
\end{lagrange1}

The Deutsch-Sagan conjecture is an immediate consequence of the following result.
\begin{theorem}
\label{t-main}
Let $F=\displaystyle\sum_{m=0}^\infty x^{3^m}$. Then 
\begin{equation*}
\sum_{n=1}^\infty N_n x^n \equiv F+F^2 \pmod3.
\end{equation*}
\end{theorem}

\begin{proof}
We apply the Lagrange inversion theorem with $G(x) =1/(1-x)(1-2x)$ and $\Phi(x) = 1/(1-x)$, so that $\Phi'(x) = 1/(1-x)^2$. Then by Lemma \ref{l-1}, together with the fact that $N_1=1$, we have
\begin{equation}
\label{e-N3}
\sum_{n=0}^\infty N_{n+1}x^n = \frac{1}{1-\f}
\end{equation}
where $\f$ is the unique formal power series satisfying
\begin{equation}
\label{e-f1}
\f=\frac{x}{(1-\f)(1-2\f)}.
\end{equation}
By \eqref{e-f1}, $\f-3\f^2 +2\f^3 =x$, so we have
\(
\f(x) \equiv x-2\f(x)^3 \equiv x+\f(x)^3 \equiv x+\f(x^3)\pmod 3
\).
Iterating this congruence gives
\begin{equation}
\label{e-mod3}
\f(x) \equiv x+x^3+\f(x^{27})\equiv \cdots \equiv \sum_{m=0}^\infty x^{3^m} \pmod 3.
\end{equation}By \eqref{e-f1}, 
\begin{equation*}
\frac{1}{1-\f} = x^{-1}(\f - 2\f^2) \equiv x^{-1} (\f+\f^2),
\end{equation*}
so by \eqref{e-N3}, 
\begin{equation}
\label{e-2f}
\sum_{n=1}^\infty N_n x^n \equiv \f+\f^2 \pmod3.
\end{equation}
Then Deutsch and Sagan's congruence \eqref{e-DS} follows directly from \eqref{e-2f} and \eqref{e-mod3}.
\end{proof}

\section{Eu, Liu, and Yeh's congruences}
\label{s-elycong}
In their proof of the Deutsch-Sagan conjecture, Eu, Liu, and Yeh \cite[Lemmas 1--4]{ely} found the residues modulo 3 for four auxiliary sequences which they define by 
\begin{align*}
f_1(n) &= \sum_i \binom{3n+1}{n+i+1}\binom in\\
f_2(n) &=\sum_i \binom{3n}{n+i+1}\binom in\\
f_3(n) &=\sum_i\binom{3n}{n+i}\binom in\\
f_4(n) &=\sum_i\binom{3n-1}{n+i+1}\binom i{n-1},
\end{align*}
with $f_4(0)=0$.
We will also consider a fifth sum 
\begin{equation*}
f_5(n) = \sum_i \binom{3n}{n+i+1}\binom{i}{n-1},
\end{equation*}
with $f_5(0)=1$.

The first few values of these sums are as follows:
\[
\offinterlineskip\vbox{\halign{\ \hfil\strut$#$\hfil\ \vrule&&\hfil\ $#$\ \cr
n&0&1&2&3&4&5&6&7&8\cr
\noalign{\hrule}
f_1(n)\strut&1&6&48&420&3840&36036&344064&3325608&32440320\cr
f_2(n)&0&1&9&82&765&7266&69930&679764&6659037\cr
f_3(n)&1&5&39&338&3075&28770&274134&2645844&25781283\cr
f_4(n)&0&1&7&58&515&4746&44758&428772&4154403\cr
f_5(n)&1&4&30&256&2310&21504&204204&1966080&19122246\cr
}}\]

Eu, Liu, and Yeh noted that $f_1(n) = f_2(n) + f_3(n)$ and that $f_2(n)$ is the number of edges in all noncrossing connected graphs on $n+1$ vertices for $n\ge1$ (sequence \seqnum{A045741}). The sequence $f_5(n)$ is sequence \seqnum{A091527} in the OEIS, and $f_1(n)$, $f_3(n)$, and $f_4(n)$ do not currently appear in the OEIS.

To derive Eu, Liu, and Yeh's congruences by our method, we consider the more general sequence $h_{j,k,l}(n)$, where  $j$, $k$, $l$, and $n$ are arbitrary integers, defined by
\begin{equation*}
h_{j,k,l}(n) 
  =\sum_{i=0}^{n+j-k+l}\binom{3n+j}{n+j-k+l-i}\binom{n-l+i}{i}.
\end{equation*}
Then if $3n+j\ge0$ and $n\ge l$, replacing the summation index $i$ with $i-n+l$ gives
\begin{equation*}
h_{j,k,l}(n) = \sum_{i=n-l}^{2n+j-k} \binom{3n+j}{n+i+k}\binom{i}{n-l}.
\end{equation*}
Thus $f_1=h_{1,1,0}$, $f_2=h_{0,1,0}$, $f_3=h_{0,0,0}$, $f_4=h_{-1, 1,1}$, and $f_5=h_{0,1,1}$.
A straightforward computation, as in the proof of Lemma \ref{l-1}, shows that 
\begin{align*}
h_{j,k,l}(n) &= [x^n] \frac{x^{k-l-j}}{(1-x)^{n+k}(1-2x)^{n-l+1}}\\
  &=[x^n] \frac{x^{k-l-j}}{(1-x)^{k}(1-2x)^{-l+1}}\cdot\frac{1}{(1-x)^n(1-2x)^n}
\end{align*}

Now we use the following form of Lagrange inversion (see, e.g., \cite[equation (4.4)]{multi}; a closely related formula is \cite[p.~150, Theorem D]{comtet}): 
\begin{lagrange2}
Let $G(x)$ be a formal power series and let $\fx=\fx(x)$ be the unique formal power series satisfying $\fx=xG(\fx)$. Then for any formal Laurent series $\Psi(x)$ and any integer $n$,
\begin{equation*}
[x^n] \frac{\Psi(\fx)}{1-xG'(\fx)} =  [x^{n}] \Psi(x)G(x)^n.
\end{equation*}
or equivalently
\begin{equation*}
[x^n] \frac{\Psi(\fx)}{1-\fx G'(\fx)/G(\fx)} =  [x^{n}] \Psi(x)G(x)^n.
\end{equation*}
\end{lagrange2}

As in the proof of Theorem \ref{t-main}, we apply this to the equation $\f=xG(\f)$, where $G(x)=1/(1-x)(1-2x)$. Then 
\begin{equation*}
\frac{1}{1-\f G'(\f)/G(\f)}=\frac{(1-\f)(1-2\f)}{1-6\f+6\f^2}.
\end{equation*}
Now let 
\begin{equation*}
H_{j,k,l}=\sum_{n=k-j-l}^\infty h_{j,k,l}(n) x^n.
\end{equation*}


Then taking 
\begin{equation*}
\Psi(x) = \frac{x^{k-j-l}}{(1-x)^{k}(1-2x)^{-l+1}}
\end{equation*} 
we obtain 
\begin{equation}
\label{e-Hjkl}
H_{j,k,l}
  =\frac{(1-2\f)^l \f^{k-j-l}}{(1-6\f+6\f^2)(1-\f)^{k-1}},
\end{equation}
and thus
\begin{equation}
\label{e-gjkl}
H_{j,k,l}
  \equiv(1+\f)^{l}(1-\f)^{1-k} \f^{k-j-l} \pmod 3.
\end{equation}
Now let $F_m=\sum_{n=0}^\infty f_m(n)x^n$. Then  by \eqref{e-gjkl} we have 
\begin{equation}
\label{e-F1-5}
\begin{gathered}
F_1 =H_{1,1,0}
\equiv  1\pmod 3\\
F_2 =H_{0,1,0}
 \equiv \f \pmod 3\\
F_3 =H_{0,0,0}
 \equiv 1-\f \pmod 3\\
F_4 =H_{-1,1,1}
 \equiv  \f+\f^2 \pmod 3\\
F_5 = H_{0,1,1}
  \equiv  1+\f\pmod 3.
\end{gathered}
\end{equation}
Then Eu, Liu, and Yeh's congruences follow immediately from these congruences and \eqref{e-mod3}. We don't state Eu, Liu, and Yeh's congruences here since they are  easy to read off from the congruences in \eqref{e-F1-5} and \eqref{e-mod3}.

\section{Evaluation of the sums}

There are simple explicit formulas for the sums $f_1$, $f_2$, $f_3$,  $f_4$, and $f_5$ (and there is a similar formula for $N_n$ that we will give in section \ref{s-more}), though these formulas don't seem to yield simpler proofs for the congruences than the proofs we have already given.

\begin{theorem}
\label{t-f}
The sequences $f_i(n)$ for $i$ from 1 to 5 are given by the following explicit formulas:
\begin{align*}
f_1(n) &= 2^{2n}\binom{\tfrac32 n}{n}\\
f_2(n) &= 2^{2n-1}\binom{\tfrac32 n}{n} - 2^{2n-1}\binom{\tfrac32n -\tfrac12}{n}\\
f_3(n) &= 2^{2n-1}\binom{\tfrac32 n}{n} + 2^{2n-1}\binom{\tfrac32n -\tfrac12}{n}\\
f_4(n) &=- \frac{2^{2n-1}}{3}\binom{\tfrac32n}{n} +2^{2n-1}\binom{\tfrac32n -\tfrac12}{n}, \text{\textnormal{ for $n>0$}}\\
f_5(n)&= 2^{2n}\binom{\tfrac32 n -\half}{n}.
\end{align*}
\end{theorem}

The evaluation of these sums is based on a binomial coefficient identity that is equivalent to a terminating case of a hypergeometric series evaluation called Kummer's theorem \cite[p.~9, Theorem 2.3]{bailey}. There are many ways to prove this identity; 
we give here a proof using Lagrange inversion. A  short self-contained proof was given by Wildon \cite{wildon}.

\begin{lemma}
\label{l-K1}
\begin{equation}
\label{e-K1}
\sum_{k=0}^n \binom{2n+a}{n-k}\binom{a+k-1}{k} = 2^{2n}\binom{\half a + n-\half}{n}
\end{equation}
\end{lemma}

\begin{proof}
Let $S$ be the  left side of \eqref{e-K1}. Then  $S$ is equal to the coefficient of $x^n$ in $1/(1-x)^{n+1}(1-2x)^a$ since
\begin{align*}
\frac{1}{(1-x)^{n+1}(1-2x)^a} &= \frac{1} {\displaystyle(1-x)^{n+a+1} \left( 1-\frac{x}{1-x}\right)^a}\\
  &= \sum_{k}  \frac{x^k}{(1-x)^{n+a+k+1}}\binom{a+k-1}{k}\\
  &=\sum_{j,k} x^{j+k}\binom{n+a+j+k}{j}  \binom{a+k-1}{k}\\
  &=\sum_m x^m \sum_k \binom{n+a+m}{m-k}\binom{a+k-1}{k}.
\end{align*}
Let us take $G(x) = 1/(1-x)$ and $\Psi(x) = 1/(1-x)(1-2x)^a$ in the second form of Lagrange inversion. 
The solution of $f=xG(f)$ is 
\begin{equation*}
f=\frac{1-\sqrt{1-4x}}{2}
\end{equation*}
and  $\Psi(x)/\bigl(1-xG'(x)/G(x)\bigr) = (1-2x)^{-a-1}$, so 
\begin{align*}
S&=[x^n]\frac{1}{(1-2f)^{a+1}}=[x^n] \frac 1{\bigl(\sqrt{1-4x}\,\bigr)^{a+1}} \\
  &=(-4)^n \binom {-(a+1)/2}{n}
  = 2^{2n}\binom{\half a + n-\half}{n}.\qedhere
\end{align*}
\end{proof}

{\allowdisplaybreaks
We can now prove Theorem \ref{t-f}. 
Let 
\begin{align*}
T_1(n,i)&=\binom{3n+1}{n+i+1}\binom in\\
T_2(n,i)&=\binom{3n}{n+i+1}\binom in\\
T_3(n,i)&=\binom{3n}{n+i}\binom in\\
T_4(n,i)&=\binom{3n-1}{n+i+1}\binom i{n-1}\\
T_5(n,i)&=\binom{3n}{n+i+1}\binom{i}{n-1},
\end{align*}
}
so that for $m=1,\dots, 5$ we have $f_m(n) = \sum_i T_m(n,i)$.
Then by Lemma \ref{l-K1} we have
\begin{gather*}
 f_1(n) = \sum_i T_1(n,i) = 2^{2n}\binom{\tfrac32 n}{n}  \\
f_5(n)=\sum_i T_5(n,i) = 2^{2n}\binom{\tfrac32 n -\half}{n}. 
\end{gather*}
Also, it is easy to check that $T_2(n,i) +T_3(n,i) = T_1(n,i)$, as noted in \cite{ely}, 
that $T_3(n,i) -T_2(n,i-1) = T_5(n,i-1)$ for $n>0$, and that
$T_4(n,i) = \tfrac 23 T_5(n,i) -\tfrac 13T_3(n, i+1)$. Thus $f_2(n) + f_3(n) = f_1(n)$,  $f_3(n) - f_2(n) = f_5(n)$, and $f_4(n) = \tfrac23 f_5(n) - \tfrac13 f_3(n)$ for $n\ge1$. We can then solve for $f_2$, $f_3$ and $f_4$ in terms of $f_1$ and $f_5$, and we can check that the formulas given in Theorem \ref{t-f} also hold for $f_m(0)$ if $m\ne 4$. \qed

%



\section{More Lagrange inversion}
\label{s-more}
We can also prove Theorem \ref{t-f}, and derive a related formula for $N_n$, by Lagrange inversion.

Let us define the power series $\g$ in $x$  by 
\begin{equation}
\label{e-g1}
\g = \frac{x}{\sqrt{1-4\g}}.
\end{equation}
If we  define the power series $\f=\f(x)$ by $\g=\f-\f^2$ in \eqref{e-g1} (together with the condition $\f(0)=0$) we see that $\f$ satisfies
\begin{equation*}
\f-\f^2 = \frac{x}{1-2\f},
\end{equation*}
so $\f=x/(1-\f)(1-2\f)$ and thus this $\f$ is the same power series as the $\f$ discussed in sections \ref{s-cong} and \ref{s-elycong}.
Applying the second form of Lagrange inversion, we have for any power series $\Psi(x)$,
\begin{equation*}
[x^n] \frac{1-4\g}{1-6\g} \Psi(\g) = [x^n]\frac{\Psi(x)}{(1-4x)^{n/2}}.
\end{equation*}
Then taking $\Psi(x) = x^i (1-4x)^{r-1}$ gives
\begin{equation*}
\frac{(1-4\g)^{r}}{1-6\g} \g^i 
  =\sum_{n=0}^\infty 2^{2n-2i}(-1)^{n-i}\binom{-\half n+r-1}{n-i}.
\end{equation*}
 Since $\g=\f-\f^2$ and $(-1)^{n-i}\binom{-n/2+r-1}{n-i}=\binom{3n/2-r-i}{n-i}$, we may write this as
\begin{equation}
\label{e-a1}
\frac{(1-2\f)^{2r}(\f-\f^2)^i}{1-6\f+6\f^2}
  =\sum_{n=0}^\infty 2^{2n-2i}\binom{\th n-r-i}{n-i} x^n.
\end{equation}
Then in the notation of section \ref{s-elycong}, by \eqref{e-Hjkl} and the formulas for the $F_m$ given in \eqref{e-F1-5}, we have 
\begin{align*}
F_1 &=\frac{1}{1-6\f+6\f^2}\\
F_2 &=\frac{\f}{1-6\f+6\f^2}\\
F_3 &=\frac{1-\f}{1-6\f+6\f^2}\\
F_4 &=\frac{\f-2\f^2}{1-6\f+6\f^2}\\
F_5 &=\frac{1-2\f}{1-6\f+6\f^2}\\
\end{align*}
from which the formulas of Theorem \ref{t-f} can be obtained: $F_1$ and $F_5$ can be evaluated by \eqref{e-a1}, $F_2$ and $F_3$ are linear combinations of $F_1$ and $F_5$, and 
$F_4 = -\frac16 F_1+\frac12 F_5 -\frac13$.
 
Similarly, the first form of Lagrange inversion gives
\begin{equation*}
[x^n] \Phi(\g) = \frac{1}{n} [x^{n-1}] \frac{\Phi'(x)}{(1-4x)^{n/2}}.
\end{equation*}
Let us take $\Phi(x) = (1-4x)^r$. Then  we have 
\begin{equation*}
(1-4\g)^r = 1+ \sum_{n=1}^\infty (-4)^{n}\frac{r}{n}\binom{-n/2+r-1}{n-1}x^n 
   =1-\sum_{n=1}^\infty 2^{2n}\frac{r}{n}\binom{\th n -r-1}{n-1}x^n,
\end{equation*}
so
\begin{equation}
\label{e-2a}
(1-2\f)^{2r}=1-\sum_{n=1}^\infty 2^{2n}\frac{r}{n}\binom{\th n -r-1}{n-1}x^n.
\end{equation}

From \eqref{e-N3} it follows that $\sum_{n=1}^\infty N_{n}x^n = x/(1-\f)$, and by \eqref{e-f1}, 
\begin{equation*}
\frac{x}{1-\f} = \f - 2\f^2 = \half(1-2\f) -\half (1-2\f)^2,
\end{equation*}
so by \eqref{e-2a}, for $n\ge1$ we have
\begin{align}
N_n &= \frac 12 \left[\frac{2^{2n}}{n} \binom{\th n -2}{n-1} -  \frac{2^{2n-1}}{n}\binom{\th n -\th}{n-1}\right]\notag\\
 &= \frac{2^{2n-1}}{n}\binom{\th n -2}{n-1} -\frac{2^{2n-2}}{n}\binom{\th n -\th}{n-1}.
 \label{e-mh}
\end{align}
An equivalent formula was stated by Mark van Hoeij in the OEIS entry for sequence \seqnum{A007297}.
The first term on the right side of \eqref{e-mh}, $({2^{2n-1}}/{n})\binom{3n/2 -2}{n-1}$ is  twice sequence \seqnum{A078531}, and the negative of the second term, $({2^{2n-2}}/{n}) \binom{3n/2-3/2}{n-1}$ is sequence \seqnum{A085614}.
We note also that if $n=2m+1$ then 
\begin{equation*}
 \frac{2^{2n-1}}{n}\binom{\th n -2}{n-1} = 2\frac{m!\,(6m)!}{(2m)!\,(2m+1)!\,(3m)!}
\end{equation*}
and if $n=2m+2$ then 
\begin{equation*}
\frac{2^{2n-2}}{n}\binom{\th n -\th}{n-1}= 6\frac{m!\,(6m+1)!}{(2m)!\,(2m+2)!\,(3m)!}.
\end{equation*}


\begin{thebibliography}{9}
\bibitem{a-s}
M. Abramowitz and I. A. Stegun, Handbook of Mathematical Functions with Formulas, Graphs, and Mathematical Tables, National Bureau of Standards Applied Mathematics Series, Vol. 55, Washington DC, 1964.




\bibitem{bailey}
W. N. Bailey, Generalized Hypergeometric Series, Hafner, New York, 1972. Originally published by Cambridge University Press, 1935.

\bibitem{comtet}
L. Comtet, Advanced Combinatorics, Reidel, Dodrecht-Holland, 1974.

\bibitem{ds}
E. Deutsch and B. Sagan,
Congruences for Catalan and Motzkin numbers and related sequences, J. Number Theory 117 (2006), 191--215.

\bibitem{ely}
S.-P. Eu, S.-C. Liu, and Y.-N. Yeh,
On the congruences of some combinatorial numbers,
Studies in Applied Math. 116 (2006), 135--144.

\bibitem{f-n}
P. Flajolet and M. Noy, Analytic combinatorics of non-crossing configurations, Discrete Math. 204, (1999) 203--229.

\bibitem{multi}
I. M. Gessel, A combinatorial proof of the multivariable Lagrange inversion formula, J. Combin. Theory Ser. A 45 (1987), 178--195.


\bibitem{oeis}
N. J. A. Sloane, The On-Line Encyclopedia of Integer Sequences,
\url{http://oeis.org}, 2013.

\bibitem{ec2}
R. P. Stanley,
 \emph{Enumerative Combinatorics}, Volume 2, Cambridge University Press, 1999.
 
\bibitem{wildon}
 M.~Wildon, Combinatorial identities,
 \url{http://mathoverflow.net/questions/150093/}, 2013.
  



\end{thebibliography}
\end{document}